\documentclass{amsart}

\usepackage{amsmath,amsthm,amssymb}

\usepackage[all]{xy}

\theoremstyle{plain}
\newtheorem{Theo}{Theorem}[section]
\newtheorem{Lemm}[Theo]{Lemma}
\newtheorem{Prop}[Theo]{Proposition}
\newtheorem{Cor}[Theo]{Corollary}

\theoremstyle{definition}

\theoremstyle{remark}

\DeclareMathOperator{\CP}{\mathbb{C}P}
\DeclareMathOperator{\uC}{\underline{\mathbb{C}}}
\DeclareMathOperator{\C}{\mathbb{C}}
\DeclareMathOperator{\Z}{\mathbb{Z}}
\DeclareMathOperator{\FH}{F _\bullet H^*}

\begin{document}
	\title[(Filtered) cohomological rigidity of Bott towers]{(Filtered) cohomological rigidity of Bott towers}
	\date{\today}
	\author[H. Ishida]{Hiroaki Ishida}
	\address{Department of Mathematics, Graduate School of Schience, Osaka City University, Sugimoto, Sumiyoshi-ku, Osaka 558-8585, Japan}
	\email{hiroaki.ishida86@gmail.com}
	\begin{abstract}
		A Bott tower is an iterated $\CP ^1$-bundle over a point, where each $\CP ^1$-bundle
		is the projectivization of a rank $2$ decomposable complex vector bundle. For a Bott tower, 
		the \emph{filtered cohomology} is naturally defined. 
		We show that isomorphism classes of Bott towers are distinguished by their filtered cohomology rings. 
		We even show that any filtered cohomology ring isomorphism between two Bott towers  
		is induced by an isomorphism of the Bott towers.
	\end{abstract}
	
	\maketitle	
	\section{Introduction}
		A \emph{Bott tower} of height $n$ is a sequence of $\CP^1$-bundles 
		\begin{equation}\label{eq:Bott tower}
			B_n\stackrel{\pi_n}\longrightarrow B_{n-1} \stackrel{\pi_{n-1}}\longrightarrow
			\dots \stackrel{\pi_2}\longrightarrow B_1 \stackrel{\pi_1}\longrightarrow
			B_0=\{\text{a point}\},
		\end{equation}
		where each fibration $\pi _i: B_i\to B_{i-1}$ for $i=1,\dots,n$ 
		is the projectivization of a rank $2$ decomposable complex vector bundle over $B_{i-1}$. 
		A decomposable complex vector bundle is the Whitney sum of line bundles.  
		Each $B_i$ is called a \emph{Bott manifold}. 
		As is known, a Bott manifold $B_n$ is a smooth projective toric variety. 
		We denote the Bott tower \eqref{eq:Bott tower} of height $n$ by 
		$B_\bullet =(\{ B_i\} _{i=0}^n , \{\pi _i\} _{i=1}^n )$.

		Two Bott towers $B_\bullet = (\{ B_i\} _{i=0}^n , \{\pi _i\} _{i=1}^n )$ and 
		$B'_\bullet =(\{ B'_i\} _{i=0}^n , \{\pi '_i\} _{i=1}^n )$ are isomorphic if there is 
		a collection of diffeomorphisms $\varphi _\bullet = \{ \varphi _i: B_i \to B'_i\} _{i=0}^n$ 
		such that the following diagram is commutative: 
			\begin{equation*}
				\xymatrix{
					B_n \ar[r]^{\pi _n} \ar[d]^{\varphi _n} & B_{n-1} \ar[r]^{\pi _{n-1}} \ar[d]^{\varphi _{n-1}}
					& \dots \ar[r]^{\pi _2} & B_1 \ar[r]^{\pi _1} \ar[d]^{\varphi _1} & B_0 \ar[d]^{\varphi _0}\\
					B'_n \ar[r]^{\pi '_n} & B'_{n-1} \ar[r]^{\pi '_{n-1}} 
					& \dots \ar[r]^{\pi '_2} & B'_1 \ar[r]^{\pi '_1} & B'_0 .
					}
			\end{equation*}

		Since $\pi _i: B_i \to B_{i-1}$ has a cross section, 
		the induced homomorphism $\pi _i^* : \mathrm{H}^*(B_{i-1}) \to \mathrm{H}^*(B_i)$ is injective for each $i$. 
		Therefore the cohomology ring $\mathrm{H}^*(B_{i-1})$ can be regarded as a subring of $\mathrm{H}^*(B_i)$ via $\pi _i^*$.  
		We define a filtered graded $\Z$-algebra $\FH (B_\bullet )$ by
		\begin{itemize}
			\item For $i \geq n$, $\mathrm{F}_i\mathrm{H}^*( B_\bullet) := \mathrm{H}^*(B_n)$.
			\item For $0 \leq j \leq n-1$, 
				$\mathrm{F}_j\mathrm{H}^*( B_\bullet ):= 
				\pi _n^*\circ \pi _{n-1}^* \circ \dots \circ \pi _{j+1}^*(\mathrm{H}^*(B_j))$.
		\end{itemize}
		We call $\FH (B_\bullet )$ the \emph{filtered cohomology ring} of the Bott tower $B_\bullet$.

		The following is our main result: 
		\begin{Theo}\label{theorem1.1}
			Isomorphism classes of Bott towers are distinguished by their filtered cohomology rings. 
			Moreover, any filtered cohomology ring isomorphism between two Bott towers are induced
			by an isomorphism between the Bott towers. 
		\end{Theo}
		Our study is motivated by the so-called \emph{cohomological rigidity problem} for toric manifolds. 
		A \emph{toric manifold} is a smooth (complete) toric variety. Bott manifolds are examples of toric manifolds.
		The cohomological rigidity problem asks whether topological types or diffeomorphism types of toric manifolds 
		are distinguished by their cohomology rings or not. This problem is open and we do not know any counterexamples.

		This paper is organized as follows.	In Section \ref{section2}, we recall some preliminary facts about Bott manifolds. 
		In Section \ref{section3}, we see that 
		rank $2$ decomposable vector bundles over Bott manifolds are distinguished by their total Chern classes. 
		In Section \ref{section4}, we study $\CP ^1$-bundles over Bott manifolds by using the results 
		obtained in Section \ref{section3}. In Section \ref{section5}, we give a proof of Theorem \ref{theorem1.1}. 
		
		Throughout this paper, all cohomology groups are taken with $\Z$-coefficient.
		\bigskip \\
		{\bf Acknowledgement}.
		The author would like to thank professor Mikiya Masuda for stimulating discussion about
toric topology. 
	\section{Preliminaries}\label{section2}
		In this section, we recall some preliminary facts (see \cite[section 2]{CM} and \cite[section 2]{CMS} for more details). 
		For a complex vector bundle $\eta $ over a topological space, we denote its projectivization by $P(\eta )$. 
		
		\begin{Lemm}[Lemma 2.1, \cite{CMS}]\label{lemma2.1}
			Let $B$ be a smooth manifold, $\xi $ be a complex line bundle over $B$ and $\eta $ be a complex vector bundle over $B$.
			Then, the projectivizations $P(\eta )$ and $P(\xi \otimes \eta )$ are isomorphic as bundles. 
		\end{Lemm}

		Let $B_{i-1}$ be a Bott manifold and $\eta _i= \xi _i \oplus \xi _i'$ be the Whitney sum of complex line bundles 
		$\xi _i$ and $\xi _i'$ over $B_{i-1}$. Since $\xi _i^* \otimes \eta _i = \uC \oplus (\xi _i^* \otimes \xi _i')$, the projectivization $B_i =P(\eta _i)$ is isomorphic to $P (\uC \oplus (\xi _i^*\otimes \xi _i'))$
		as fiber bundles by Lemma \ref{lemma2.1}, where $\uC $ denotes the trivial line bundle over $B_{i-1}$ and $\xi _i^*$ denotes the dual line bundle of $\xi _i$. 
		Hence we can assume that $\eta _i$ is the Whitney sum of $\uC $ and a complex line bundle, denoted $\xi _i$ again, in the definition of Bott towers. 
		
		Let $B_\bullet =(\{ B_i\} _{i=0}^n , \{\pi _i\} _{i=1}^n )$ be a Bott tower of height $n$. 
		We suppose that each fibration $B_i \to B_{i-1}$ is the projectivization $P(\uC \oplus \xi _i)$. 
		Let $\gamma _i$ be the tautological line bundle over $P(\uC \oplus \xi _i) =B_i$. 
		By Borel-Hirzebruch formula, the cohomology ring $\mathrm{H}^*(B_i)$ viewed as an $\mathrm{H}^*(B_{i-1})$-algebra via $\pi _i^*$ is of the form
		\begin{equation*}\label{BorelHirzebruch}
			\mathrm{H}^*(B_i) \cong \mathrm{H}^*(B_{i-1})[X]/(X^2-c_1(\xi _i)X)
		\end{equation*}
		where $c_1(\xi _i)$ denotes the first Chern class of $\xi _i$ 
		and $X$ represents the first Chern class of the tautological line bundle $\gamma _i$.
		Using this formula inductively on $i$, we see that the cohomology ring $\mathrm{H} ^*(B_n)$ is generated by 
		\begin{equation*}
			x_i := \pi _n^*\circ \dots \circ \pi _{i+1}^*(c_1(\gamma _i)) 
		\end{equation*}
		for $i = 1,\dots ,n$. More precisely, $\mathrm{H} ^*(B_n)$ is of the form
		\begin{equation*}
			\mathrm{H} ^*(B_n) = \Z [x_1,\dots ,x_n]/(x_i^2 - \alpha _ix_i ; i=1,\dots ,n)
		\end{equation*}
		as rings, 
		where $\alpha _i := \pi _n^*\circ \dots \circ \pi _{i+1}^*(c_1(\xi _i))$ for all $i$. 
		
		With this understood, we may regard $\mathrm{H} ^*(B_k)$ as a subring of $\mathrm{H}^*(B_n)$ 
		generated by $x_j$'s with $j =1,\dots ,k$. Namely, we have 
		\begin{equation*}
			\mathrm{F}_k\mathrm{H}^*(B_\bullet ) = \Z [x_1,\dots ,x_k]/(x_j^2 - \alpha _jx_j ; j=1,\dots ,k). 
		\end{equation*}
		
		A \emph{vanishing pair} is an ordered pair $(z,\overline{z})$ of elements in $\mathrm{H}^2(B_n)$ such that $z\overline{z} = 0$. 
		We say that a vanishing pair is \emph{primitive} if both elements are primitive. The following lemma is very helpful for our purpose. 
		\begin{Lemm}[Lemma 2.3, \cite{CM}]\label{pvp}
			A primitive vanishing pair is of the form
			\begin{equation*}
				(ax_j+u, \pm (a(x_j-\alpha _j)-u))
			\end{equation*}
			for some $j$, where $a$ is a non-zero integer, $u$ is a linear combination of $x_i$'s with $i < j$, and $u(u+a\alpha _j)=0$.  
		\end{Lemm}

	\section{Rank $2$ decomposable vector bundles over Bott manifolds}\label{section3}
		For an element $\alpha \in \mathrm{H}^2(B_n)$, we denote by $\gamma ^\alpha $ the complex line bundle over a Bott manifold $B_n$ 
		whose first Chern class is $\alpha$. 
		The purpose of this section is to prove the following. 
		\begin{Theo}\label{theorem3.1}
			Rank $2$ decomposable vector bundles over a Bott manifold  
			are distinguished by their total Chern classes. 
		\end{Theo}		
		We prepare two lemmas.
		\begin{Lemm}\label{lemma3.2}
				$\gamma ^{x_j}\oplus \gamma ^{-x_j+ \alpha _j} \cong \uC \oplus \gamma ^{\alpha _j}$.
			\end{Lemm}
			\begin{proof}
				Since $\gamma ^{x_j}$ is the pull-back bundle of the tautological line bundle over $P(\uC \oplus \xi _j)=B_j$, 
				$\gamma ^{x_j}$ is a subbundle of $\uC \oplus \gamma ^{\alpha _j}$ where $\alpha _j = c_1(\xi _j)$. It follows that the first Chern class of the orthogonal complement of $\gamma ^{x_j}$ in $\uC \oplus \gamma ^{\alpha _j}$ (with an Hermitian metric) is 
				$ -x_j + \alpha _j$. Since the orthogonal complement is a complex line bundle
				and complex line bundles are classified by their first Chern classes, it must be isomorphic to $\gamma ^{-x_j +\alpha _j}$, proving the lemma. 
			\end{proof}
			\begin{Lemm}\label{lemma3.3}
				If $\gamma ^\alpha \oplus \gamma ^\beta \cong \uC \oplus \gamma ^{\alpha + \beta }$,  
				then $\gamma ^{a\alpha } \oplus \gamma ^{b\beta } \cong \uC \oplus \gamma ^{a\alpha + b\beta }$ for any integers $a$ and $b$.  
			\end{Lemm}
			\begin{proof}
				There is a nowhere zero cross section $f$ of $\gamma ^\alpha \oplus \gamma ^\beta$ since $\gamma ^\alpha \oplus \gamma ^\beta $ contains a trivial complex line bundle as a subbundle by assumption. Let $f_\alpha $ and $f_\beta$ be the projections of $f$ on $\gamma ^\alpha$ and $\gamma ^\beta $, which are cross sections of $\gamma ^\alpha$ and $\gamma ^\beta $ respectively. Their zero loci do not intersect. If $a$ (respectively, $b$) is positive, we define $f_{a\alpha } := f_\alpha ^a = f_\alpha \otimes \dots \otimes f_\alpha$ (respectively, $f_{b\beta } := f_\beta ^b$) which is a cross section of $\gamma ^{a\alpha }$ (respectively, $\gamma ^{b\beta }$). 
				Note that there is a map $\varphi _{a\alpha } : \gamma ^{a\alpha } \to \gamma ^{-a\alpha }$ 
				(respectively, $\varphi _{b\beta} : \gamma ^{b\beta } \to \gamma ^{-b\beta}$) which is an isomorphism as 
				real $2$-plane bundles but anti-$\mathbb{C}$-linear on fibers.
				If $a$ (respectively, $b$) is negative, we define a cross section $f_{a\alpha }$ of $\gamma ^{a\alpha }$ 
				(respectively, $f_{b\beta }$ of $\gamma ^{b\beta }$) 
				to be the pull-back of $f_\alpha ^{-a}$ (respectively, $f_\beta ^{-b}$) by $\varphi _{a\alpha } : \gamma ^{a\alpha } \to \gamma ^{-a\alpha}$ (respectively, $\varphi _{b\beta} : \gamma ^{b\beta } \to \gamma ^{-b\beta}$). The pair $(f_{a\alpha }, f_{b\beta })$ determines a nowhere zero cross section of $\gamma ^{a\alpha }\oplus \gamma ^{b\beta}$. 
				Thus $\gamma ^{a\alpha } \oplus \gamma ^{b\beta }$ contains a trivial complex line bundle as a subbundle. 
				By computing the first Chern class of the orthogonal complement of the trivial subbundle in 
				$\gamma ^{a\alpha } \oplus \gamma ^{b\beta }$, we get the lemma. 
			\end{proof}
		\begin{Prop}\label{proposition3.4}
			Let $\xi $ and $\xi '$ be complex line bundles over $B_n$. Then, the second Chern class $c_2( \xi \oplus \xi ')$ is zero if and only if $\xi \oplus \xi'$ is isomorphic to $\uC \oplus (\xi \otimes \xi ')$. 
		\end{Prop}
		\begin{proof}
			The \lq\lq{if}'' part is clear. So we show the \lq\lq{only if}'' part. 
			
			We assume that $c_2(\xi \oplus \xi ' )= 0$. Then $c_1(\xi )c_1(\xi ') =0$. 
			Note that if either $c_1(\xi )=0$ or $c_1(\xi ')=0$, then $\xi $ or $\xi '$ is trivial and Proposition \ref{proposition3.4} holds. Thus, by lemma \ref{lemma3.3} and the above note, it suffices to show that $\xi \oplus \xi ' \cong \C \oplus (\xi \otimes \xi ')$ only when the vanishing pair $(c_1(\xi _i),c_1(\xi _i'))$ is primitive.

			By Lemma \ref{pvp} and \ref{lemma3.3}, we may assume
			\begin{equation*}
				(c_1(\xi ),c_1(\xi '))=(ax_j+u, a(x_j-\alpha _j)-u)
			\end{equation*}
			for some $j$, where $a$ is a non-zero integer, $u$ is a linear combination of $x_i$'s with $i < j$, and $u(u+a\alpha _j) = 0$. 
			Let $j$ be the minimal integer such that $c_1(\xi )$ and $c_1(\xi ')$ can be written as linear combinations of $x_k$'s with $k \leq j$. 
			We will show that $\xi  \oplus \xi ' \cong \uC \oplus (\xi \otimes \xi ')$ by induction on $j$. 
			If $j=1$, then $(c_1(\xi ),c_1(\xi ')) = \pm ( x_1, x_1)$ and we are done by Lemma \ref{lemma3.2} and \ref{lemma3.3}. Now we assume that $\gamma ^\alpha \oplus \gamma ^{\alpha '} \cong \uC \oplus (\gamma ^\alpha \otimes \gamma ^{\alpha '})$ for any vanishing pair $(\alpha , \alpha ')$ such that $\alpha$ and $\alpha '$ can be written as linear combinations of $x_k$'s with $k < j$. 
			Then, 
			\begin{equation*}
				\begin{split}
					\gamma ^{ax_j+u}\oplus \gamma ^{a(x_j-\alpha _j)-u} & \cong \gamma ^{ax_j} \otimes (\gamma ^u \oplus \gamma ^{-(u+a\alpha _j)})  \\ 
					& \cong \gamma ^{ax_j} \otimes (\uC \oplus \gamma ^{-a\alpha _j}) \text{ (by hypothesis of induction)} \\
					& \cong \gamma ^{ax_j} \oplus \gamma ^{a(x_j - \alpha _j)}  \\
					& \cong \uC \oplus \gamma ^{a(2x_j -\alpha _j)} \text{ (by Lemma \ref{lemma3.2} and \ref{lemma3.3})}.
				\end{split}
			\end{equation*}
			This completes the induction step, proving the proposition.
		\end{proof}
		\begin{proof}[Proof of Theorem \ref{theorem3.1}]
			Let $\alpha $, $\beta $, $\alpha '$ and $\beta '$ be elements in $\mathrm{H}^2(B_n)$ such that 
			\begin{equation*}
				c(\gamma ^\alpha \oplus \gamma ^\beta ) = c(\gamma ^{\alpha '} \oplus \gamma ^{\beta '}), 
			\end{equation*}
			that is, $(1+\alpha )(1+\beta )= (1+\alpha ')(1+\beta ')$. Then, we have
			\begin{equation*}
				\beta - \alpha = (\alpha ' -\alpha ) + (\beta ' - \alpha ) 
			\hspace*{5mm}\text{ and }\hspace*{5mm}
				(\alpha ' -\alpha )(\beta ' -\alpha )=0.
			\end{equation*}
			Therefore, it follows from Proposition \ref{proposition3.4} that 
			\begin{equation*}
				\uC \oplus \gamma ^{\beta - \alpha } \cong \gamma ^{\alpha ' -\alpha } \oplus \gamma ^{\beta ' -\alpha }.
			\end{equation*}
			By tensoring the both sides above with $\gamma ^\alpha $, 
			we obtain the theorem. 
		\end{proof}
	\section{$\CP ^1$-bundles over Bott manifolds}\label{section4}
		In this section, we study $\CP ^1$-bundles over Bott manifolds. 
		\begin{Theo}\label{theorem4.1}
			Let $B_n$ be a Bott manifold. Let $\gamma ^\alpha $ and $\gamma ^\beta $ be complex line bundles over $B_n$. 
			If $\mathrm{H}^*(P(\uC \oplus \gamma ^\alpha ))$ and $\mathrm{H}^*(P(\uC \oplus \gamma ^\beta ))$ are isomorphic as $\mathrm{H}^*(B_n)$-algebras,
			then $P(\uC \oplus \gamma ^\alpha )$ and $P(\uC \oplus \gamma ^\beta )$ are isomorphic as bundles. 
		\end{Theo}
		\begin{proof}
			By Borel-Hirzebruch formula, $\mathrm{H}^*(P(\uC \oplus \gamma ^\alpha ))$ and 
			$\mathrm{H}^*(P(\uC \oplus \gamma ^\beta ))$ are of the forms 
			\begin{equation*}
				\mathrm{H}^*(P(\uC \oplus \gamma ^\alpha )) = \mathrm{H}^*(B_k)[X]/(X^2-\alpha X)
			\end{equation*}
			and 
			\begin{equation*}
				\mathrm{H}^*(P(\uC \oplus \gamma ^\beta )) = \mathrm{H}^*(B_k)[Y]/(Y^2-\beta Y).
			\end{equation*}
			Let $\Phi : \mathrm{H}^*(P(\uC \oplus \gamma ^\beta )) \to \mathrm{H}^*(P(\uC \oplus \gamma ^\alpha ))$ be an $\mathrm{H}^*(B_n)$-algebra isomorphism. We write $\Phi (Y) =: sX+ \alpha '$, $s=\pm 1$, $\alpha ' \in \mathrm{H}^2(B_k)$. Since $\Phi (Y^2 -\beta Y)=0$ and 
			$X^2=\alpha X$, we have
			\begin{equation*}
				\begin{split}
					0 &= (sX+\alpha ')^2 - \beta (sX+\alpha ') \\
					&=(\alpha + 2s\alpha ' - s\beta )X + \alpha '(\alpha '-\beta ).
				\end{split}
			\end{equation*}
			Therefore $\alpha = s(-2\alpha '+ \beta )$ and $\alpha '(\alpha '-\beta ) =0$. Thus we obtain 
			\begin{equation*}
				\begin{split}
					P(\uC \oplus \gamma ^\alpha ) &= P(\uC \oplus \gamma ^{s(-2\alpha '+\beta )})\\
					&\cong P(\uC \oplus \gamma ^{-2\alpha '+\beta } ) \text{ (by taking complex conjugation when $s=-1$)}\\
					&\cong P(\gamma ^{\alpha '} \oplus \gamma ^{-\alpha '+ \beta }) \text{ (by Lemma \ref{lemma2.1})}\\
					&\cong P(\uC \oplus \gamma ^\beta ) \text{ (by Theorem \ref{theorem3.1})}.
				\end{split}
			\end{equation*}
			This proves the theorem.
		\end{proof}
		\begin{Lemm}\label{lemma4.2}
			The group of $\mathrm{H}^*(B_n)$-algebra automorphisms of $\mathrm{H}^*(P(\uC \oplus \gamma ^\alpha ))$ is of order $2$ and
			generated by the automorphism induced by the bundle map $p : P(\uC \oplus \gamma ^\alpha ) \to P(\uC \oplus \gamma ^\alpha )$ 
			which assigns a line $\ell $ in $\uC \oplus \gamma ^\alpha$ (with an Hermitian metric) to its orthogonal complement $\ell ^\perp$.
		\end{Lemm}
		\begin{proof}
		Let $\Phi : \mathrm{H}^*(P(\uC \oplus \gamma ^\alpha )) \to \mathrm{H}^*(P(\uC \oplus \gamma ^\alpha ))$ be an $\mathrm{H}^*(B_n)$-algebra automorphism. 
		The same argument as in the proof of Theorem \ref{theorem4.1} shows that $\Phi (X) = X$ or $-X +\alpha$. 
		Thus the group of $\mathrm{H}^*(B_n)$-algebra automorphisms of $\mathrm{H}^*(P(\uC \oplus \gamma ^\alpha ))$ is of order $2$. 
		We have to show that $p^*(X) = -X+\alpha $. 
		Note that the total space of the tautological line bundle $\gamma ^X$ over $P(\uC \oplus \gamma ^\alpha )$ is the set
		\begin{equation*}
			\{ (\ell , v) \in P(\uC \oplus \gamma ^\alpha )\times E(\uC \oplus \gamma ^\alpha ); \ell \ni v\}
		\end{equation*}
		and the total space of the pull-back $p^*\gamma ^X$ of $\gamma ^X$ by $p$ is given by  
		\begin{equation*}
			\{ (\ell , v) \in P(\uC \oplus \gamma ^\alpha )\times E(\uC \oplus \gamma ^\alpha ); \ell ^\perp \ni v\} ,
		\end{equation*}
		so $p^*\gamma ^X  \oplus \gamma ^X  \cong \pi ^*(\uC \oplus \gamma ^\alpha )$, where $\pi : P(\uC \oplus \gamma ^\alpha ) \to B_n$ is the projection. 
		Therefore we obtain $p^*(X) = -X+\alpha $. 
		\end{proof}
		The following corollary follows from Theorem \ref{theorem4.1} and Lemma \ref{lemma4.2}.
		\begin{Cor}\label{corollary4.3}
			Let $B_n$ be a Bott manifold. Let $\eta $ and $\eta '$ be 
			rank $2$ decomposable vector bundles over $B_n$. Then, their projectivizations $P(\eta )$ and $P(\eta ')$ are isomorphic as bundles
			if and only if
			their cohomology rings are isomorphic as $\mathrm{H}^*(B_n)$-algebras. Moreover, any $\mathrm{H}^*(B_n)$-algebra isomorphism between $\mathrm{H}^*(P(\eta ))$ and $\mathrm{H}^*(P(\eta '))$ is induced by a bundle isomorphism. 
		\end{Cor}
	\section{Proof of Theorem \ref{theorem1.1}}\label{section5}
	Let $B_\bullet = (\{ B_i\} _{i=0}^n , \{\pi _i\} _{i=1}^n )$ and $B'_\bullet =(\{ B'_i\} _{i=0}^n , \{\pi '_i\} _{i=1}^n )$ 
	be Bott towers of height $n$. Let $\Phi _\bullet : \FH (B_\bullet ')  \to \FH (B_\bullet )$ be an isomorphism. 
	We will show Theorem 1.1 by induction on height. Assume Theorem \ref{theorem1.1} holds for Bott towers of height $k$. 
	For $\Phi _k : \mathrm{F}_k\mathrm{H}^* (B_\bullet ')=\mathrm{H}^*(B_k') \to \mathrm{H}^*(B_k) = \mathrm{F}_k\mathrm{H}^*( B_\bullet )$, 
	there is a diffeomorphism $\varphi _k : B_k \to B_k '$ such that $\varphi _k^* = \Phi _k$ by the hypothesis of induction. Then, the pull-back bundle $\varphi _k^*B_{k+1}' \to B_k$ is isomorphic to $B_{k+1} \to B_k$. In fact, there is a commutative diagram 
	\begin{equation*}
		\xymatrix{
			\mathrm{H}^*(B_{k+1}) & & \mathrm{H}^*(B_{k+1}') \ar[ll]_{\Phi _{k+1}} \ar[ld]_{\widetilde{\varphi _k}^*} \\
			 & \mathrm{H}^*(\varphi _k^*B_{k+1}') & \\
			\mathrm{H}^*(B_k) \ar[uu]^{\pi _{k+1}^*} \ar[ru]^{(\varphi _k^*\pi _{k+1}')^*} & & 
			\mathrm{H}^*(B_k') \ar[ll]^{\varphi _k^* = \Phi _k} \ar[uu]_{\pi _{k+1}'^*}
		}
	\end{equation*}
	where $\widetilde{\varphi _k}^*$  is the homomorphism induced from the bundle map 
	$\widetilde{\varphi _k} : \varphi _k^*B_{k+1}' \to B_{k+1}'$ and $(\varphi _k^*\pi _{k+1}')^*$ 
	is the homomorphism induced from the projection 
	$\varphi ^*\pi _{k+1}' : \varphi _k^*B_{k+1}' \to B_k$. 
	Since $\widetilde{\varphi _k}$ is a diffeomorphism, the induced homomorphism $\widetilde{\varphi _k}^*$ is a ring isomorphism. 
	The composition $\Phi _{k+1}\circ (\widetilde{\varphi _k}^*)^{-1} : \mathrm{H}^*(\varphi _k^*B_{k+1}') \to \mathrm{H}^*(B_{k+1})$
	is an $\mathrm{H}^*(B_k)$-algebra isomorphism. By Corollary \ref{corollary4.3}, there is a bundle isomorphism $b : B_{k+1} \to \varphi _k^*B_{k+1}'$ such that
	$b^* = \Phi _{k+1}\circ (\widetilde{\varphi _k}^*)^{-1}$. 
	
	By the construction of $b$, the bundle map $\varphi _{k+1}:= \widetilde{\varphi _k} \circ b$ satisfies $\varphi _{k+1}^* = \Phi _{k+1}$.   
	\qed
	
\end{document}